\documentclass[11pt]{amsart}
\usepackage{amsmath,amsthm,amssymb}
\usepackage[latin1]{inputenc}
\usepackage[T1]{fontenc}
\usepackage[all]{xy}
\usepackage{mathrsfs}
\numberwithin{equation}{section}
\newtheorem{thm}[equation]{Theorem}
\newtheorem*{thma*}{Theorem A}
\newtheorem*{thmb*}{Theorem B}
\newtheorem*{thmc*}{Theorem C}
\newtheorem*{thmd*}{Theorem D}

\newtheorem{lem}[equation]{Lemma}

\theoremstyle{definition}

\newtheorem{rem}[section]{Remark}

\newtheorem*{acknow}{Acknowledgement}

\def\Z{\mathbb Z}
\def\A{\mathbb A}
\def\Q{\mathbb Q}
\def\C{\mathbb C}

\def\E{\mathbb E}

\def\<{\langle}
\def\>{\rangle}

\def\GL{{\rm GL}}

\usepackage{bm}
\usepackage{upgreek}
\makeatletter
\newcommand{\bfgreek}[1]{\bm{\@nameuse{up#1}}}
\makeatother
\begin{document}

%opening
\title[Motivic Periods]{On Deligne's periods for tensor product motives}
\author{\bf Chandrasheel Bhagwat}
\date{\today}
\address{Chandrasheel Bhagwat, Indian Institute of Science Education and Research, Dr.Homi Bhabha Road, Pashan, Pune 411008, INDIA.}
\email{cbhagwat@iiserpune.ac.in}

%\keywords{cuspidal automorphic representation, Whitaker models,  cuspidal cohomology, period, $L$-function, algebraicity result}
\thanks{The author is partially supported by the DST-INSPIRE Faculty scheme, award number [IFA- 11MA -05].}

\begin{abstract}

In this paper, we give a description of Deligne's periods $c^\pm$ for tensor product of pure motives $M \otimes M'$ over $\Q$ in terms of the period invariants attached to $M$ and $M'$ by Yoshida~\cite{Yo}. The period relations proved by the author and Raghuram in an earlier paper follow from the results of this paper.

\end{abstract}

\maketitle

\section{\bf Introduction}

Let $M$ be a pure motive over $\Q$ with coefficients in a number field $\Q(M).$ Suppose $M$ is critical, then
a celebrated conjecture of Deligne~\cite[Conj.\,2.8]{deligne} relates the critical values of its $L$-function $L(s,M)$ to certain periods  $c^{\pm}(M(\Pi))$ which are defined through a comparison of the Betti and de~Rham realizations of the motive.

Conjecturally, one can associate a motive $M(\Pi)$ to a given cohomological cuspidal automorphic representation $\Pi$ of $GL_{n}(\A_{\Q})$. One expects from this correspondence that the standard $L$-function $L(s, \Pi)$ is the motivic $L$-function $L(s, M(\Pi))$ up to a shift in the $s$-variable; see Clozel \cite[Sect.\,4]{clozel}. There are certain periods $p^\epsilon(\Pi)$ which have been defined by Raghuram-Shahidi \cite{raghuram-shahidi-imrn}. Given cohomological cuspidal automorphic representations $\Pi$ and $\Sigma$ of $GL_{n}(\A_{\Q})$ and $GL_{n-1}(\A_{\Q})$ respectively, Raghuram (\cite{raghuram-imrn}, \cite{raghuram-2014}) has proved that the product $p^\epsilon(\Pi)$ $p^\eta(\Sigma)$, for a suitable choice of signs $\epsilon$ and $\eta$, appears in the critical values of the Rankin-Selberg $L$-function $L(s,~\Pi \times \Sigma)$. One can ask whether there is an analogous relation for the Deligne periods so that the results of \cite{raghuram-2014} are compatible with Deligne's conjecture.

\smallskip

In this paper, we give a description of Deligne's periods $c^\pm(M \otimes M')$ for the tensor product $M \otimes M'$, where $M$ and $M'$ are two pure motives over $\Q$ all of whose nonzero Hodge numbers are one, in terms of the periods $c^\pm(M)$, $c^\pm(M')$ {\it and} some other finer invariants attached to $M$ and $M'$ by Yoshida~\cite{Yo}. The main period relations are in
Thm.\,\ref{PR1}, Thm.\,\ref{PR2} and Thm.\,\ref{PR3}. The period relations for the ratio $\frac{c^{+}(M \otimes M')}{c^{-}(M \otimes M')}$  proved by the author and Raghuram in \cite{bhagwat-raghuram} follow from these results.

\smallskip

\section{\bf Preliminaries}\label{prelim}

\subsection{Critical motives}

Let $M$ be a motive defined over $\Q$ with coefficients in a number field $\mathbb{E}$. Let $H_{B}(M)$ be the
{\it Betti realization} of $M$. It is a finite-dimensional vector space over $\E$. The rank $d(M)$ of $M$ is defined to be $\text{dim}_{\E}H_{B}(M)$. Write $H_{B}(M) \ = \ H^{+}_{B}(M) \oplus H^{-}_{B}(M),$
where $H^{\pm}_{B}(M)$ are the $\pm 1$-eigenspaces for the action of complex conjugation $\rho$ on
$H_{B}(M)$. Let $d^{\pm}(M)$ be the $\E$-dimensions of $H^{\pm}_{B}(M)$. The Betti realization has a
{\it Hodge decomposition}:
\begin{equation}
\label{eqn:hodge-decomp}
H_{B}(M) \otimes_{\Q} \C \ = \ \bigoplus_{p,q \in \Z} H^{p,q}(M),
\end{equation}
where $H^{p,q}(M)$ is a free $\E \otimes \C$-module of rank $h^{p,q}_M$. The numbers $h^{p,q}_M$ are called the {\it Hodge numbers} of
$M.$  We say that $M$ is {\it pure} if there is an integer $w$ (which is called the purity weight of $M$) such that
 $H^{p,q}(M) = \left\{0\right\}$ if $p+q \neq w.$ Henceforth, we assume that all the motives we consider are pure.
We also have $\rho(H^{p,q}(M)) = H^{q,p}(M);$ and hence $\rho$ acts on the (possibly zero) middle Hodge type $H^{w/2,w/2}(M).$

\medskip

Let  $H_{DR}(M)$ be the {\it de~Rham realization} of $M$; it is a $d(M)$-dimensional vector space over $\mathbb{E}$. There is a comparison isomorphism of $\E \otimes_\Q \C$-modules:
\[
I : H_{B}(M) \otimes_{\Q} \C  \ \longrightarrow \ H_{DR}(M) \otimes_{\Q} \C.
\]
The de~Rham realization has a {\it Hodge filtration} $F^p(M)$ which is a decreasing filtration of $\E$-subspaces of $H_{DR}(M)$ such that $I \left(\oplus_{p' \geq p} H^{p',w-p'}(M)\right) = F^{p}(M) \otimes_{\Q} \C.$
Write the Hodge filtration as
\begin{equation}
\label{Hodge}
H_{DR}(M) = F^{p_1}(M) \supsetneq F^{p_2}(M) \supsetneq \cdots \supsetneq F^{p_m}(M) \supsetneq F^{p_{m+1}}(M) = \left\{0\right\};
\end{equation}
all the inclusions are proper and there are no other filtration-pieces between two successive members.  We assume that the numbers $p_\mu$ are maximal among all the choices.
Let $s_\mu = h^{p_\mu, w-p_\mu}_M$ for $1 \leq \mu \leq m.$ Purity plus the action of complex conjugation on Hodge types says that the numbers $p_j$  and $s_\mu$ satisfy
$p_{j} + p_{m+1-j} = w, \forall\  1 \leq j \leq m,$ and
$s_\mu = s_{m+1-\mu}, \forall\  1 \leq \mu \leq m.$

\medskip

We say that the motive {\it $M$ is critical} if there exist $p^+, p^- \in \Z$ such that
$ \sum_{i=1}^{p^{+}} s_i= d^+(M)$ and $\sum_{i=1}^{p^{-}} s_i= d^-(M).$
In this case one says that $F^{\pm}(M)$ exists and equals $F^{{p}^{\pm}}(M)$.

\medskip

\subsection{Tensor product of motives}

Let $M$ and $M'$ be pure motives defined over $\Q$ and with coefficients in a number field $\E$. Suppose that their ranks are $n$ and $n'$ and purity weights are $w$ and $w'$, resp. We further assume that all the non-zero Hodge numbers of $M$ and $M'$ equal to $1$.\smallskip

Suppose $H_{B}(M) \otimes \C =  \oplus_{j=1}^{n} H^{p_{j},~ w-p_{j}}(M),$
where $p_{j}$ are integers such that $p_{1} < p_{2} < \ldots < p_{n}$. Similarly, suppose
$H_{B}(M') \otimes \C =  \bigoplus \limits_{j=1}^{n'} H^{q_{j}, ~w'-q_{j}}(M'),$
with $q_{1} < q_{2} < \ldots < q_{n'}$.

\smallskip

Since all the non-zero Hodge numbers of $M$ and $M'$ equal to $1$, it follows that the Hodge filtrations of the de~Rham realizations of $M$, $M'$ and $M \otimes M'$ are given by
\[ H_{DR}(M) =  F^{p_{1}}(M) \supsetneq F^{p_{2}}(M) \supsetneq \ldots \supsetneq F^{p_{n}}(M) \supsetneq (0).\]
\[ H_{DR}(M') =  F^{q_{1}}(M') \supsetneq F^{q_{2}}(M') \supsetneq \ldots \supsetneq F^{q_{n'}}(M') \supsetneq (0).\]
\[ H_{DR}(M \otimes M') =  F^{r_{1}}(M \otimes M') \supsetneq F^{r_{2}}(M \otimes M') \supsetneq \ldots \supset F^{r_{m}}(M \otimes M') \supsetneq (0).\]

\smallskip

Let $u_t$ denote the dimension of $ F^{r_{t}}(M \otimes M')/ F^{r_{t+1}}(M \otimes M')$ for $1 \leq t \leq m$. Let us further assume that $M \otimes M'$ is critical. Consider the complex conjugation action on Betti realizations for the motives $M$ and $M'$. \smallskip

If the dimension $nn'$ is an even integer, it follows that $d^{\pm}(M \otimes M')$ are equal to $\frac{nn'}{2}$. From the criticality of $M \otimes M'$, it follows that there is $k^{+} = k^{-} = k_{0} \geq 1$ such that,
\[ u_1 + u_2 + \ldots + u_{k_0} = d^{\pm}(M \otimes M') = \frac{nn'}{2}.\]
Let $1 \leq i \leq n$ and $ 1 \leq j \leq n'$. Following Yoshida \cite{Yo}, we define:
$$
a_i = |\left\{j : 1 \leq j \leq n' : p_i + q_j \leq r_{k_{0}}  \right\}|,\quad \quad a^{\ast}_j = |\left\{i : 1 \leq i \leq n : p_i + q_j \leq r_{k_{0}}  \right\}|.
$$

\smallskip

If $nn'$ is odd, there exist $k^{+}$, $k^{-}$ such that
\[ u_1 + u_2 + \ldots + u_{k^+} \quad = \quad d^{+}(M \otimes M') \quad = \quad \frac{nn' \pm 1}{2},\]
\[ u_1 + u_2 + \ldots + u_{k^{-}} \quad = \quad d^{-}(M \otimes M') \quad = \quad  \frac{nn' \mp 1}{2}.\]

\medskip

It follows that $k^+ - k^- =  d^{+}(M \otimes M') - d^{-}(M \otimes M') = \pm 1$.
Let $1 \leq i \leq n$ and $ 1 \leq j \leq n'$. In this case we define:
$$
a^{\pm}_i = |\left\{j : 1 \leq j \leq n' : p_i + q_j \leq r_{k^{\pm}}  \right\}|,\quad \quad a^{\ast, \pm}_j = |\left\{i : 1 \leq i \leq n : p_i + q_j \leq r_{k^{\pm}}  \right\}|.
$$

\subsection{Invariant polynomials and periods}
\label{sec:yoshida-polynomials}
The period matrix of $M$ is defined in terms of $\E$-bases for the spaces $H^{\pm}_{B}(M)$ and $H_{DR}(M)$.
Let $\left\{v_{1},v_{2},\ldots,v_{d^{+}(M)}\right\}$ be an $\E$-basis of $H^{+}_{B}(M)$, and similarly,
$\left\{v_{d^{+}(M) + 1},v_{d^{+}(M) + 2},\ldots,v_{d(M)}\right\}$ be an $\E$-basis of $H^{-}_{B}(M)$. Let $\left\{w_{1},w_{2},\ldots,w_{d(M)}\right\}$ be a basis of $H_{DR}(M)$ over $\E$ such that $\left\{w_{s_{1}+s_{2}+ \ldots +s_{\mu -1}+1},\ldots,w_{d(M)}\right\}$ is a basis of $F^{p_\mu}(M)$ for $1 \leq \mu \leq m$.
The period matrix $X$ of $M$ is the matrix which represents the comparison isomorphism between the two realizations of $M$ with respect to the bases chosen above. The fundamental periods $c^{\pm}(M)$ and $\delta(M)$ are related to the matrix $X$ through certain invariant polynomials.

\smallskip

Let $\mathbb{F}$ be a number field. Suppose $d$ is a positive integer. Fix a partition $s_1 + s_2 + \ldots + s_m = d$. Let $P_m$ be the corresponding lower parabolic subgroup of $\GL(d)$. Given an $m$-tuple of integers $(a_i)_{1 \leq i \leq m}$, define an algebraic character $\lambda_1$ of $P_m$ by
\[
\lambda_1 \left( {\begin{pmatrix} p_{11}& 0 & \ldots & 0 \\ \ast & p_{22} & \ldots & 0\\\ast & \ast & \ddots & \ldots \\ \ast & \ast & \ast & p_{mm} \end{pmatrix}} \right) = \prod \limits_{1 \leq i \leq m}(\text{det}~p_{ii})^{a_i};  \quad p_{ii} \in \GL(s_i).
\]
Let $d = d^+ + d^-$. Given $k^+, k^- \in \Z$, define a character $\lambda_2$ of $\GL(d^+) \times  \GL(d^-)$ by
\[
\lambda_2 \left( \begin{pmatrix} a & 0 \\ 0 & b \end{pmatrix} \right) = ~(\text{det}~a)^{k^+} (\text{det}~b)^{k^-}, \quad a \in \GL(d^+),~ b \in \GL(d^-).
\]

Let $f(x)$ be a polynomial with rational coefficients which satisfies the following equivariance condition with respect to the left action of $P_m$ and the right action of $\GL(d^+) \times \GL(d^-)$ on the matrix ring $M_{d}(F)$:
\begin{equation}
\label{inv}
f(p x \gamma) = \lambda_1(p) f(x) \lambda_2({\gamma}),  \quad \forall~ p \in P_m, \ \
\forall~ \gamma \in \GL(d^+) \times \GL(d^-).
\end{equation}
A polynomial satisfying (\ref{inv})
is said to have admissibility type $\left\{(a_1, a_2, \ldots, a_m), (k^+, k^-)\right\}$.
Yoshida \cite[Theorem 1]{Yo} proves that  the space of polynomials of a given admissibility type is at most one.

 \begin{lem}\label{multi}
 If the polynomial $f(x)$ has admissibility type $\left\{(a_1, a_2, \ldots, a_m), (k_1^+, k_1^-)\right\}$,
 and $g(x)$  has admissibility type  $\left\{(b_1, b_2, \ldots, b_m), (k_2^+, k_2^-)\right\}$, then the polynomial $h(x) = f(x)g(x)$ has admissible type is given by
 \[
 \left\{(a_1 + b_1, a_2 + b_2, \ldots, a_m + b_m), (k_1^+ + k_2^+, k_1^- + k_2^-)\right\}.
 \]
 \end{lem}

\begin{proof} Follows from (\ref{inv}).
\end{proof}

The admissibility type of $f(x) = {\rm det}(x)$ for $x \in M_{d}(F)$, is
$\left\{(1,1,1,\ldots,1),(1,1)\right\}$.
Let $f^{\pm}(x)$ be the upper left  (resp., upper right) $d^{\pm} \times d^{\pm}$ determinant of $x$.
 Then it can be seen that the admissibility types of $f^+(x)$ and $f^-(x)$ are respectively given by
 \[
 \{(\underbrace{1,1,1,\ldots,1}_{p^{+}},0,0, \ldots, 0),(1,0)\}, \quad
 \{(\underbrace{1,1,1,\ldots,1}_{p^{-}},0,0, \ldots, 0),(0,1)\}.
 \]

Yoshida interprets the period invariants to the period matrix $X$ via invariant polynomials as $\delta(M) = f(X)$ and $c^{\pm}(M) = f^{\pm}(X)$. The determinant of the period matrix is an element of $(\E \otimes \C)^\times$, and making a choice of basis says that it is well-defined modulo $\E^{\times}$.

\medskip

\section{\bf{Calculation of motivic periods $c^{\pm}(M \otimes M')$ }}

One knows from the results of Yoshida \cite{Yo} that the motivic periods $c^{\pm}(M \otimes M')$ can be expressed as monomials in the other period invariants $\delta(M)$, $c^{\pm}(M)$, $c^{\pm}(M')$,
$c_{p}(M)$, $c_{p}(M')$, ($p$ runs over a finite set). For the definitions of these period invariants, see \cite{Yo}. In this section we calculate these monomials explicitly. \smallskip

First we consider the case where the ranks of motives have opposite parities. Let $M$ and $M'$ be motives defined over $\Q$ with coefficients in $\mathbb{E}$ as in section \ref{prelim} with ranks $n = 2k$ and $n' = 2k' + 1$, resp. We set $\epsilon(M') : = d^{+}(M') - d^{-}(M') = \pm 1.$ Thus $d^{\pm}(M) = k$, $d^{+}(M') = k'+ \frac{(\epsilon(M')+1)}{2}$ and $d^{-}(M') = k'+\frac{(1-\epsilon(M'))}{2}$. Let's define two finite sets by,
$$\mathcal{P} = \left\{ 1,~2,~\ldots, k-1 = ~\text{min} \left\{d^{\pm}(M) \right\} -1 \right\}, $$ $$\mathcal{P'} = \left\{ 1,~2,~\ldots, k' - 1 = ~\text{min} \left\{d^{\pm}(M') \right\} - 1 \right\}. $$\smallskip

Consider the expression for $c^{+}(M \otimes M')$ as a monomial in other period invariants with integer exponents as follows:

\begin{equation} \label{1}
c^{+}(M \otimes M') = \delta(M)^{\alpha}~ \delta(M')^{\beta}~ c^{+}(M)^{\alpha^{+}}~ c^{-}(M)^{\alpha^{-}} ~ c^{+}(M')^{\beta^{+}}~
c^{-}(M')^{\beta^{-}} ~  \prod \limits_{p \in \mathcal{P} } c_{p}(M)^{\alpha_{p}} \prod \limits_{p \in \mathcal{P}} c_{p}(M')^{\beta_p}.
\end{equation}

\medskip

We have a similar expression for $c^{-}(M \otimes M')$.
From Yoshida \cite{Yo}, we know that admissibility types for the period invariants $\delta(M)$, $\delta(M')$, $c^{\pm}(M)$, $c^{\pm}(M')$,
$c_{p}(M)$, $c_{p}(M')$ are given by:\\

$\delta(M) : (\underbrace{1,~1,~\ldots,~1}_{n ~\text{times}}),~(1,1)$,\quad $\delta(M') : (\underbrace{1,~1,~\ldots,~1}_{n' ~\text{times}}),~(1,1),$ \\ \smallskip

$c^{+}(M) : (\underbrace{1,~1,~\ldots,~1}_{d^{+}(M)~ \text{times}},~\underbrace{0,~0,~\ldots,~0}_{d^{-}(M)~ \text{times}}),~(1,~0)$,\quad  $c^{-}(M) : (\underbrace{1,~1,~\ldots,~1}_{d^{-}(M)~ \text{times}},~\underbrace{0,~0,~\ldots,~0}_{d^{+}(M)~ \text{times}}),~(0,~1),$ \\ \smallskip

$c^{+}(M') : (\underbrace{1,~1,~\ldots,~1}_{d^{+}(M')~ \text{times}},~\underbrace{0,~0,~\ldots,~0}_{d^{-}(M')~ \text{times}}),~(1,~0)$, \quad $c^{-}(M') : (\underbrace{1,~1,~\ldots,~1}_{d^{-}(M')~ \text{times}},~\underbrace{0,~0,~\ldots,~0}_{d^{+}(M')~ \text{times}}),~(0,~1),$\\ \smallskip

$c_{p}(M) : (\underbrace{2,~2,~\ldots,~2}_{p \ \text{times}},~\underbrace{1,~1,~\ldots,~1}_{n-2p\ \text{times}},~\underbrace{0,~0,~\ldots,~0}_{p \ \text{times}}),~ (1,~1) \quad \forall \quad p \in \mathcal{P},$ \ {\rm and} \\ \smallskip

$c_{p}(M') : (\underbrace{2,~2,~\ldots,~2}_{p \ \text{times}},~\underbrace{1,~1,~\ldots,~1}_{n'-2p\ \text{times}},~\underbrace{0,~0,~\ldots,~0}_{p \ \text{times}}), ~(1,~1) \quad \forall \quad p \in \mathcal{P'}.$ \\ \medskip

Let $X$ and $Y$ be the period matrices corresponding to the comparison isomorphism between Betti and de~Rham realizations. Then $c^{\pm}(M \otimes M')$ equals the product of $ \phi^{\pm}(X)~ \psi^{\pm}(Y)$ and their admissibility types are:
$$
\phi^{\pm}(X) : (a_1,~a_2,~\ldots~a_n), ~(d^{\pm}(M'), ~d^{\mp}(M')),  \quad {\rm and} \quad
\psi^{\pm}(Y) : (a^{\ast}_1,~a^{\ast}_2,~\ldots~a^{\ast}_{n'}),~ (d^{\pm}(M), ~d^{\mp}(M)).
$$
Using the admissibility types and (\ref{1}) we get:
$$
 \alpha = a_{n},\quad \alpha^{+} = a_k - n'/2 + \epsilon(M')/2, \quad  \alpha^{-} = a_k - n'/2 - \epsilon(M')/2, \quad  \alpha_{p} = a_{p} - a_{p+1} \quad \forall  \ p \in \mathcal{P},
$$
$$
\beta = a^{\ast}_{n'},\quad \beta^{+} = \beta^{-} = a^{\ast}_{k'} - n/2, \quad
\beta_{p} = a^{\ast}_{p} - a^{\ast}_{p+1} \quad \forall  \ p \in \mathcal{P'}.
$$

\medskip

\begin{thm} \label{PR1}
If the ranks of $M$ and $M'$ are even and odd respectively and $M \otimes M'$ is critical, then the periods $c^{\pm} (M \otimes M')$ are given by:
\begin{equation*}
c^{+} (M \otimes M') = T c^{\epsilon(M')}(M), \quad c^{-} (M \otimes M') = T c^{-\epsilon(M')}(M),
\end{equation*}
where $T$ is defined as:
\begin{multline*}
T = \delta(M)^{a_{n}} ~\delta(M')^{a^{\ast}_{n'}}~ [c^{+}(M)c^{-}(M)]^{a_{k} - k' - 1}~  [c^{+}(M')c^{-}(M')]^{a^{\ast}_{k'} - k - 1}
\\
\cdot \prod \limits_{p \in \mathcal{P}} c_{p}(M)^{a_{p} - a_{p+1}} ~ \prod \limits_{p \in \mathcal{P'}} c_{p}(M')^{a^{\ast}_{p} - a^{\ast}_{p+1}}.
\end{multline*}
\end{thm}
\medskip
The $\pm$ sign in the exponent of $c^{\pm}(M)$ period in the above expression is determined by the sign of $\epsilon(M')$. This in particular is consistent with the following result of \cite{bhagwat-raghuram}.
\begin{equation} \label{rel1}
\frac{c^{+}(M \otimes M')}{c^{-}(M \otimes M')} =  \left( \frac{c^{+}(M)}{c^{-}(M)} \right)^{\epsilon(M')}.
\end{equation}
\smallskip

The case when both $M$ and $M'$ have same parity can be handled in an exactly analogous manner. We consider two cases: \medskip

\textbf{Case 1:} Both $M$ and $M'$ have odd ranks $n = 2k + 1$, $n' = 2k' +1,$ respectively.
From the definition of the integers $a^{\pm}_{i}$ and $a^{\ast,\pm}_{i}$ it follows that
\[
a^{+}_{i} = a^{-}_{i} \quad  \forall ~ 1 \leq i \leq ~\text{min}\left\{d^{\pm}(M)\right\}, \quad {\rm and} \quad
a^{*,+}_{j} = a^{*,-}_{j} \quad \forall ~ 1 \leq j \leq ~\text{min}\left\{d^{\pm}(M')\right\}.
\]

\medskip
\begin{thm} \label{PR2}
If both $M$ and $M'$ are of odd rank and $M \otimes M'$ is critical, then
\begin{equation*}
c^{+} (M \otimes M') = T c^{\epsilon(M')}(M) c^{\epsilon(M)}(M'), \quad
c^{-} (M \otimes M') = T c^{-\epsilon(M')}(M) c^{-\epsilon(M)}(M'),
\end{equation*}
where the period $T$ is defined by the same formula as in Thm.\,\ref{PR1}; but note that $k$ and $\mathcal{P}$
(resp., $k'$ and $\mathcal{P}'$) depend on $n$ (resp., $n'$).
\end{thm}

\medskip

The following period relation from \cite{bhagwat-raghuram} is an easy consequence:
\begin{equation} \label{rel2}
\frac{c^{+} (M \otimes M') }{c^{-} (M \otimes M') } =  \left( \frac{c^{+}(M)}{c^{-}(M)} \right)^{\epsilon(M')}  \left( \frac{c^{+}(M')}{c^{-}(M')} \right)^{\epsilon(M)}.
\end{equation}

\bigskip

\textbf{Case 2:} If both $M$ and $M'$ have even ranks then it turns out that $c^{+} (M \otimes M') = c^{-} (M \otimes M').$  Set $k = n/2$ and $k' = n'/2$, then we have:
\medskip
 \begin{thm}  \label{PR3}
 If both $M$ and $M'$ are of even rank and $M \otimes M'$ is critical, then
 \begin{multline*}
c^{\pm} (M \otimes M')
 = \delta(M)^{a_{n}}~ \delta(M')^{a^{\ast}_{n'}} ~\left[c^{+}(M)c^{-}(M)\right]^{a_{k} - k'}
 [c^{+}(M')c^{-}(M')]^{a^{\ast}_{k'} - k} \\
\cdot \prod \limits_{p \in \mathcal{P}} c_{p}(M)^{a_{p} - a_{p+1}}~  \prod \limits_{p \in \mathcal{P'}} c_{p}(M')^{a^{\ast}_{p} - a^{\ast}_{p+1}}.
\end{multline*}
\end{thm}

\smallskip

\begin{rem} A comparison of Thm.1.1 in \cite{raghuram-2014} with
our Thm.\,\ref{PR1} shows that the periods $p^\epsilon(\Pi)$ and $c^\pm(M(\Pi))$ are very different. On the other hand, an interesting question is whether there is an automorphic analogue of our period relations. In the particular case when the base field is an imaginary quadratic extension of
$\Q$, and $\Pi$ is a base change from a unitary group, there is a period relation by Grobner--Harris~\cite[Thm.\,6.7]{grobner-harris}. In general, it seems to be a hard question. The real problem seems to be to identify the finer invariants of Yoshida, denoted $c_{p}(M)$ earlier, when
the motive $M = M(\pi)$ corresponds to a cohomological cuspidal representation $\Pi$, to some kind of
automorphic invariants defined entirely in terms of $\Pi$.
\end{rem}
\smallskip

\begin{acknow}
It is a pleasure to thank Prof. Raghuram for suggesting this problem and for all the helpful discussions on this subject.

 We also thank the anonymous referee for providing constructive comments and help in improving the contents of this article. \end{acknow}

\end{document}